\documentclass[11pt]{article}
\usepackage{amsthm}
\newtheorem{theorem}{Theorem}[section]
\newtheorem{lemma}[theorem]{Lemma}
\newtheorem{corollary}[theorem]{Corollary}

\makeindex \makeglossary
\begin{document}
%
%

\long\def\ig#1{\relax}
\ig{Thanks to Roberto Minio for this def'n.  Compare the def'n of
\comment in AMSTeX.}

\newcount \coefa
\newcount \coefb
\newcount \coefc
\newcount\tempcounta
\newcount\tempcountb
\newcount\tempcountc
\newcount\tempcountd
\newcount\xext
\newcount\yext
\newcount\xoff
\newcount\yoff
\newcount\gap%
\newcount\arrowtypea
\newcount\arrowtypeb
\newcount\arrowtypec
\newcount\arrowtyped
\newcount\arrowtypee
\newcount\height
\newcount\width
\newcount\xpos
\newcount\ypos
\newcount\run
\newcount\rise
\newcount\arrowlength
\newcount\halflength
\newcount\arrowtype
\newdimen\tempdimen
\newdimen\xlen
\newdimen\ylen
\newsavebox{\tempboxa}%
\newsavebox{\tempboxb}%
\newsavebox{\tempboxc}%

\makeatletter
\setlength{\unitlength}{.01em}%
\def\settypes(#1,#2,#3){\arrowtypea#1 \arrowtypeb#2 \arrowtypec#3}
\def\settoheight#1#2{\setbox\@tempboxa\hbox{#2}#1\ht\@tempboxa\relax}%
\def\settodepth#1#2{\setbox\@tempboxa\hbox{#2}#1\dp\@tempboxa\relax}%
\def\settokens[#1`#2`#3`#4]{%
     \def\tokena{#1}\def\tokenb{#2}\def\tokenc{#3}\def\tokend{#4}}
\def\setsqparms[#1`#2`#3`#4;#5`#6]{%
\arrowtypea #1
\arrowtypeb #2
\arrowtypec #3
\arrowtyped #4
\width #5
\height #6
}
\def\setpos(#1,#2){\xpos=#1 \ypos#2}

\def\bfig{\begin{picture}(\xext,\yext)(\xoff,\yoff)}
\def\efig{\end{picture}}

\def\putbox(#1,#2)#3{\put(#1,#2){\makebox(0,0){$#3$}}}

\def\settriparms[#1`#2`#3;#4]{\settripairparms[#1`#2`#3`1`1;#4]}%

\def\settripairparms[#1`#2`#3`#4`#5;#6]{%
\arrowtypea #1
\arrowtypeb #2
\arrowtypec #3
\arrowtyped #4
\arrowtypee #5
\width #6
\height #6
}

\def\resetparms{\settripairparms[1`1`1`1`1;500]\width 500}

\resetparms

\def\mvector(#1,#2)#3{
\put(0,0){\vector(#1,#2){#3}}%
\put(0,0){\vector(#1,#2){30}}%
}
\def\evector(#1,#2)#3{{
\arrowlength #3
\put(0,0){\vector(#1,#2){\arrowlength}}%
\advance \arrowlength by-30
\put(0,0){\vector(#1,#2){\arrowlength}}%
}}

\def\horsize#1#2{%
\settowidth{\tempdimen}{$#2$}%
#1=\tempdimen
\divide #1 by\unitlength
}

\def\vertsize#1#2{%
\settoheight{\tempdimen}{$#2$}%
#1=\tempdimen
\settodepth{\tempdimen}{$#2$}%
\advance #1 by\tempdimen
\divide #1 by\unitlength
}

\def\vertadjust[#1`#2`#3]{%
\vertsize{\tempcounta}{#1}%
\vertsize{\tempcountb}{#2}%
\ifnum \tempcounta<\tempcountb \tempcounta=\tempcountb \fi
\divide\tempcounta by2
\vertsize{\tempcountb}{#3}%
\ifnum \tempcountb>0 \advance \tempcountb by20 \fi
\ifnum \tempcounta<\tempcountb \tempcounta=\tempcountb \fi
}

\def\horadjust[#1`#2`#3]{%
\horsize{\tempcounta}{#1}%
\horsize{\tempcountb}{#2}%
\ifnum \tempcounta<\tempcountb \tempcounta=\tempcountb \fi
\divide\tempcounta by20
\horsize{\tempcountb}{#3}%
\ifnum \tempcountb>0 \advance \tempcountb by60 \fi
\ifnum \tempcounta<\tempcountb \tempcounta=\tempcountb \fi
}

\ig{ In this procedure, #1 is the paramater that sticks out all the way,
#2 sticks out the least and #3 is a label sticking out half way.  #4 is
the amount of the offset.}

\def\sladjust[#1`#2`#3]#4{%
\tempcountc=#4
\horsize{\tempcounta}{#1}%
\divide \tempcounta by2
\horsize{\tempcountb}{#2}%
\divide \tempcountb by2
\advance \tempcountb by-\tempcountc
\ifnum \tempcounta<\tempcountb \tempcounta=\tempcountb\fi
\divide \tempcountc by2
\horsize{\tempcountb}{#3}%
\advance \tempcountb by-\tempcountc
\ifnum \tempcountb>0 \advance \tempcountb by80\fi
\ifnum \tempcounta<\tempcountb \tempcounta=\tempcountb\fi
\advance\tempcounta by20
}

\def\putvector(#1,#2)(#3,#4)#5#6{{%
\xpos=#1
\ypos=#2
\run=#3
\rise=#4
\arrowlength=#5
\arrowtype=#6
\ifnum \arrowtype<0
    \ifnum \run=0
        \advance \ypos by-\arrowlength
    \else
        \tempcounta \arrowlength
        \multiply \tempcounta by\rise
        \divide \tempcounta by\run
        \ifnum\run>0
            \advance \xpos by\arrowlength
            \advance \ypos by\tempcounta
        \else
            \advance \xpos by-\arrowlength
            \advance \ypos by-\tempcounta
        \fi
    \fi
    \multiply \arrowtype by-1
    \multiply \rise by-1
    \multiply \run by-1
\fi
\ifnum \arrowtype=1
    \put(\xpos,\ypos){\vector(\run,\rise){\arrowlength}}%
\else\ifnum \arrowtype=2
    \put(\xpos,\ypos){\mvector(\run,\rise)\arrowlength}%
\else\ifnum\arrowtype=3
    \put(\xpos,\ypos){\evector(\run,\rise){\arrowlength}}%
\fi\fi\fi
}}

\def\putsplitvector(#1,#2)#3#4{
\xpos #1
\ypos #2
\arrowtype #4
\halflength #3
\arrowlength #3
\gap 140
\advance \halflength by-\gap
\divide \halflength by2
\ifnum \arrowtype=1
    \put(\xpos,\ypos){\line(0,-1){\halflength}}%
    \advance\ypos by-\halflength
    \advance\ypos by-\gap
    \put(\xpos,\ypos){\vector(0,-1){\halflength}}%
\else\ifnum \arrowtype=2
    \put(\xpos,\ypos){\line(0,-1)\halflength}%
    \put(\xpos,\ypos){\vector(0,-1)3}%
    \advance\ypos by-\halflength
    \advance\ypos by-\gap
    \put(\xpos,\ypos){\vector(0,-1){\halflength}}%
\else\ifnum\arrowtype=3
    \put(\xpos,\ypos){\line(0,-1)\halflength}%
    \advance\ypos by-\halflength
    \advance\ypos by-\gap
    \put(\xpos,\ypos){\evector(0,-1){\halflength}}%
\else\ifnum \arrowtype=-1
    \advance \ypos by-\arrowlength
    \put(\xpos,\ypos){\line(0,1){\halflength}}%
    \advance\ypos by\halflength
    \advance\ypos by\gap
    \put(\xpos,\ypos){\vector(0,1){\halflength}}%
\else\ifnum \arrowtype=-2
    \advance \ypos by-\arrowlength
    \put(\xpos,\ypos){\line(0,1)\halflength}%
    \put(\xpos,\ypos){\vector(0,1)3}%
    \advance\ypos by\halflength
    \advance\ypos by\gap
    \put(\xpos,\ypos){\vector(0,1){\halflength}}%
\else\ifnum\arrowtype=-3
    \advance \ypos by-\arrowlength
    \put(\xpos,\ypos){\line(0,1)\halflength}%
    \advance\ypos by\halflength
    \advance\ypos by\gap
    \put(\xpos,\ypos){\evector(0,1){\halflength}}%
\fi\fi\fi\fi\fi\fi
}

\def\putmorphism(#1)(#2,#3)[#4`#5`#6]#7#8#9{{%
\run #2
\rise #3
\ifnum\rise=0
  \puthmorphism(#1)[#4`#5`#6]{#7}{#8}{#9}%
\else\ifnum\run=0
  \putvmorphism(#1)[#4`#5`#6]{#7}{#8}{#9}%
\else
\setpos(#1)%
\arrowlength #7
\arrowtype #8
\ifnum\run=0
\else\ifnum\rise=0
\else
\ifnum\run>0
    \coefa=1
\else
   \coefa=-1
\fi
\ifnum\arrowtype>0
   \coefb=0
   \coefc=-1
\else
   \coefb=\coefa
   \coefc=1
   \arrowtype=-\arrowtype
\fi
\width=2
\multiply \width by\run
\divide \width by\rise
\ifnum \width<0  \width=-\width\fi
\advance\width by60
\if l#9 \width=-\width\fi
\putbox(\xpos,\ypos){#4}
{\multiply \coefa by\arrowlength
\advance\xpos by\coefa
\multiply \coefa by\rise
\divide \coefa by\run
\advance \ypos by\coefa
\putbox(\xpos,\ypos){#5} }%
{\multiply \coefa by\arrowlength
\divide \coefa by2
\advance \xpos by\coefa
\advance \xpos by\width
\multiply \coefa by\rise
\divide \coefa by\run
\advance \ypos by\coefa
\if l#9%
   \put(\xpos,\ypos){\makebox(0,0)[r]{$#6$}}%
\else\if r#9%
   \put(\xpos,\ypos){\makebox(0,0)[l]{$#6$}}%
\fi\fi }%
{\multiply \rise by-\coefc
\multiply \run by-\coefc
\multiply \coefb by\arrowlength
\advance \xpos by\coefb
\multiply \coefb by\rise
\divide \coefb by\run
\advance \ypos by\coefb
\multiply \coefc by70
\advance \ypos by\coefc
\multiply \coefc by\run
\divide \coefc by\rise
\advance \xpos by\coefc
\multiply \coefa by140
\multiply \coefa by\run
\divide \coefa by\rise
\advance \arrowlength by\coefa
\ifnum \arrowtype=1
   \put(\xpos,\ypos){\vector(\run,\rise){\arrowlength}}%
\else\ifnum\arrowtype=2
   \put(\xpos,\ypos){\mvector(\run,\rise){\arrowlength}}%
\else\ifnum\arrowtype=3
   \put(\xpos,\ypos){\evector(\run,\rise){\arrowlength}}%
\fi\fi\fi}\fi\fi\fi\fi}}

\def\puthmorphism(#1,#2)[#3`#4`#5]#6#7#8{{%
\xpos #1
\ypos #2
\width #6
\arrowlength #6
\putbox(\xpos,\ypos){#3\vphantom{#4}}%
{\advance \xpos by\arrowlength
\putbox(\xpos,\ypos){\vphantom{#3}#4}}%
\horsize{\tempcounta}{#3}%
\horsize{\tempcountb}{#4}%
\divide \tempcounta by2
\divide \tempcountb by2
\advance \tempcounta by30
\advance \tempcountb by30
\advance \xpos by\tempcounta
\advance \arrowlength by-\tempcounta
\advance \arrowlength by-\tempcountb
\putvector(\xpos,\ypos)(1,0){\arrowlength}{#7}%
\divide \arrowlength by2
\advance \xpos by\arrowlength
\vertsize{\tempcounta}{#5}%
\divide\tempcounta by2
\advance \tempcounta by20
\if a#8 %
   \advance \ypos by\tempcounta
   \putbox(\xpos,\ypos){#5}%
\else
   \advance \ypos by-\tempcounta
   \putbox(\xpos,\ypos){#5}%
\fi}}

\def\putvmorphism(#1,#2)[#3`#4`#5]#6#7#8{{%
\xpos #1
\ypos #2
\arrowlength #6
\arrowtype #7
\settowidth{\xlen}{$#5$}%
\putbox(\xpos,\ypos){#3}%
{\advance \ypos by-\arrowlength
\putbox(\xpos,\ypos){#4}}%
{\advance\arrowlength by-140
\advance \ypos by-70
\ifdim\xlen>0pt
   \if m#8%
      \putsplitvector(\xpos,\ypos){\arrowlength}{\arrowtype}%
   \else
      \putvector(\xpos,\ypos)(0,-1){\arrowlength}{\arrowtype}%
   \fi
\else
   \putvector(\xpos,\ypos)(0,-1){\arrowlength}{\arrowtype}%
\fi}%
\ifdim\xlen>0pt
   \divide \arrowlength by2
   \advance\ypos by-\arrowlength
   \if l#8%
      \advance \xpos by-40
      \put(\xpos,\ypos){\makebox(0,0)[r]{$#5$}}%
   \else\if r#8%
      \advance \xpos by40
      \put(\xpos,\ypos){\makebox(0,0)[l]{$#5$}}%
   \else
      \putbox(\xpos,\ypos){#5}%
   \fi\fi
\fi
}}

\def\topadjust[#1`#2`#3]{%
\yoff=10
\vertadjust[#1`#2`{#3}]%
\advance \yext by\tempcounta
\advance \yext by 10
}
\def\botadjust[#1`#2`#3]{%
\vertadjust[#1`#2`{#3}]%
\advance \yext by\tempcounta
\advance \yoff by-\tempcounta
}
\def\leftadjust[#1`#2`#3]{%
\xoff=0
\horadjust[#1`#2`{#3}]%
\advance \xext by\tempcounta
\advance \xoff by-\tempcounta
}
\def\rightadjust[#1`#2`#3]{%
\horadjust[#1`#2`{#3}]%
\advance \xext by\tempcounta
}
\def\rightsladjust[#1`#2`#3]{%
\sladjust[#1`#2`{#3}]{\width}%
\advance \xext by\tempcounta
}
\def\leftsladjust[#1`#2`#3]{%
\xoff=0
\sladjust[#1`#2`{#3}]{\width}%
\advance \xext by\tempcounta
\advance \xoff by-\tempcounta
}
\def\adjust[#1`#2;#3`#4;#5`#6;#7`#8]{%
\topadjust[#1``{#2}]
\leftadjust[#3``{#4}]
\rightadjust[#5``{#6}]
\botadjust[#7``{#8}]}

\def\putsquarep<#1>(#2)[#3;#4`#5`#6`#7]{{%
\setsqparms[#1]%
\setpos(#2)%
\settokens[#3]%
\puthmorphism(\xpos,\ypos)[\tokenc`\tokend`{#7}]{\width}{\arrowtyped}b%
\advance\ypos by \height
\puthmorphism(\xpos,\ypos)[\tokena`\tokenb`{#4}]{\width}{\arrowtypea}a%
\putvmorphism(\xpos,\ypos)[``{#5}]{\height}{\arrowtypeb}l%
\advance\xpos by \width
\putvmorphism(\xpos,\ypos)[``{#6}]{\height}{\arrowtypec}r%
}}

\def\putsquare{\@ifnextchar <{\putsquarep}{\putsquarep%
   <\arrowtypea`\arrowtypeb`\arrowtypec`\arrowtyped;\width`\height>}}
\def\square{\@ifnextchar< {\squarep}{\squarep
   <\arrowtypea`\arrowtypeb`\arrowtypec`\arrowtyped;\width`\height>}}
\def\squarep<#1>[#2`#3`#4`#5;#6`#7`#8`#9]{{
\setsqparms[#1]
\xext=\width                                          
\yext=\height                                         
\topadjust[#2`#3`{#6}]
\botadjust[#4`#5`{#9}]
\leftadjust[#2`#4`{#7}]
\rightadjust[#3`#5`{#8}]
\begin{picture}(\xext,\yext)(\xoff,\yoff)
\putsquarep<\arrowtypea`\arrowtypeb`\arrowtypec`\arrowtyped;\width`\height>%
(0,0)[#2`#3`#4`#5;#6`#7`#8`{#9}]%
\end{picture}%
}}

\def\putptrianglep<#1>(#2,#3)[#4`#5`#6;#7`#8`#9]{{%
\settriparms[#1]%
\xpos=#2 \ypos=#3
\advance\ypos by \height
\puthmorphism(\xpos,\ypos)[#4`#5`{#7}]{\height}{\arrowtypea}a%
\putvmorphism(\xpos,\ypos)[`#6`{#8}]{\height}{\arrowtypeb}l%
\advance\xpos by\height
\putmorphism(\xpos,\ypos)(-1,-1)[``{#9}]{\height}{\arrowtypec}r%
}}

\def\putptriangle{\@ifnextchar <{\putptrianglep}{\putptrianglep
   <\arrowtypea`\arrowtypeb`\arrowtypec;\height>}}
\def\ptriangle{\@ifnextchar <{\ptrianglep}{\ptrianglep
   <\arrowtypea`\arrowtypeb`\arrowtypec;\height>}}

\def\ptrianglep<#1>[#2`#3`#4;#5`#6`#7]{{
\settriparms[#1]%
\width=\height                         
\xext=\width                           
\yext=\width                           
\topadjust[#2`#3`{#5}]
\botadjust[#3``]
\leftadjust[#2`#4`{#6}]
\rightsladjust[#3`#4`{#7}]
\begin{picture}(\xext,\yext)(\xoff,\yoff)
\putptrianglep<\arrowtypea`\arrowtypeb`\arrowtypec;\height>%
(0,0)[#2`#3`#4;#5`#6`{#7}]%
\end{picture}%
}}

\def\putqtrianglep<#1>(#2,#3)[#4`#5`#6;#7`#8`#9]{{%
\settriparms[#1]%
\xpos=#2 \ypos=#3
\advance\ypos by\height
\puthmorphism(\xpos,\ypos)[#4`#5`{#7}]{\height}{\arrowtypea}a%
\putmorphism(\xpos,\ypos)(1,-1)[``{#8}]{\height}{\arrowtypeb}l%
\advance\xpos by\height
\putvmorphism(\xpos,\ypos)[`#6`{#9}]{\height}{\arrowtypec}r%
}}

\def\putqtriangle{\@ifnextchar <{\putqtrianglep}{\putqtrianglep
   <\arrowtypea`\arrowtypeb`\arrowtypec;\height>}}
\def\qtriangle{\@ifnextchar <{\qtrianglep}{\qtrianglep
   <\arrowtypea`\arrowtypeb`\arrowtypec;\height>}}

\def\qtrianglep<#1>[#2`#3`#4;#5`#6`#7]{{
\settriparms[#1]
\width=\height                         
\xext=\width                           
\yext=\height                          
\topadjust[#2`#3`{#5}]
\botadjust[#4``]
\leftsladjust[#2`#4`{#6}]
\rightadjust[#3`#4`{#7}]
\begin{picture}(\xext,\yext)(\xoff,\yoff)
\putqtrianglep<\arrowtypea`\arrowtypeb`\arrowtypec;\height>%
(0,0)[#2`#3`#4;#5`#6`{#7}]%
\end{picture}%
}}

\def\putdtrianglep<#1>(#2,#3)[#4`#5`#6;#7`#8`#9]{{%
\settriparms[#1]%
\xpos=#2 \ypos=#3
\puthmorphism(\xpos,\ypos)[#5`#6`{#9}]{\height}{\arrowtypec}b%
\advance\xpos by \height \advance\ypos by\height
\putmorphism(\xpos,\ypos)(-1,-1)[``{#7}]{\height}{\arrowtypea}l%
\putvmorphism(\xpos,\ypos)[#4``{#8}]{\height}{\arrowtypeb}r%
}}

\def\putdtriangle{\@ifnextchar <{\putdtrianglep}{\putdtrianglep
   <\arrowtypea`\arrowtypeb`\arrowtypec;\height>}}
\def\dtriangle{\@ifnextchar <{\dtrianglep}{\dtrianglep
   <\arrowtypea`\arrowtypeb`\arrowtypec;\height>}}

\def\dtrianglep<#1>[#2`#3`#4;#5`#6`#7]{{
\settriparms[#1]
\width=\height                         
\xext=\width                           
\yext=\height                          
\topadjust[#2``]
\botadjust[#3`#4`{#7}]
\leftsladjust[#3`#2`{#5}]
\rightadjust[#2`#4`{#6}]
\begin{picture}(\xext,\yext)(\xoff,\yoff)
\putdtrianglep<\arrowtypea`\arrowtypeb`\arrowtypec;\height>%
(0,0)[#2`#3`#4;#5`#6`{#7}]%
\end{picture}%
}}

\def\putbtrianglep<#1>(#2,#3)[#4`#5`#6;#7`#8`#9]{{%
\settriparms[#1]%
\xpos=#2 \ypos=#3
\puthmorphism(\xpos,\ypos)[#5`#6`{#9}]{\height}{\arrowtypec}b%
\advance\ypos by\height
\putmorphism(\xpos,\ypos)(1,-1)[``{#8}]{\height}{\arrowtypeb}r%
\putvmorphism(\xpos,\ypos)[#4``{#7}]{\height}{\arrowtypea}l%
}}

\def\putbtriangle{\@ifnextchar <{\putbtrianglep}{\putbtrianglep
   <\arrowtypea`\arrowtypeb`\arrowtypec;\height>}}
\def\btriangle{\@ifnextchar <{\btrianglep}{\btrianglep
   <\arrowtypea`\arrowtypeb`\arrowtypec;\height>}}

\def\btrianglep<#1>[#2`#3`#4;#5`#6`#7]{{
\settriparms[#1]
\width=\height                         
\xext=\width                           
\yext=\height                          
\topadjust[#2``]
\botadjust[#3`#4`{#7}]
\leftadjust[#2`#3`{#5}]
\rightsladjust[#4`#2`{#6}]
\begin{picture}(\xext,\yext)(\xoff,\yoff)
\putbtrianglep<\arrowtypea`\arrowtypeb`\arrowtypec;\height>%
(0,0)[#2`#3`#4;#5`#6`{#7}]%
\end{picture}%
}}

\def\putAtrianglep<#1>(#2,#3)[#4`#5`#6;#7`#8`#9]{{%
\settriparms[#1]%
\xpos=#2 \ypos=#3
{\multiply \height by2
\puthmorphism(\xpos,\ypos)[#5`#6`{#9}]{\height}{\arrowtypec}b}%
\advance\xpos by\height \advance\ypos by\height
\putmorphism(\xpos,\ypos)(-1,-1)[#4``{#7}]{\height}{\arrowtypea}l%
\putmorphism(\xpos,\ypos)(1,-1)[``{#8}]{\height}{\arrowtypeb}r%
}}

\def\putAtriangle{\@ifnextchar <{\putAtrianglep}{\putAtrianglep
   <\arrowtypea`\arrowtypeb`\arrowtypec;\height>}}
\def\Atriangle{\@ifnextchar <{\Atrianglep}{\Atrianglep
   <\arrowtypea`\arrowtypeb`\arrowtypec;\height>}}

\def\Atrianglep<#1>[#2`#3`#4;#5`#6`#7]{{
\settriparms[#1]
\width=\height                         
\xext=\width                           
\yext=\height                          
\topadjust[#2``]
\botadjust[#3`#4`{#7}]
\multiply \xext by2 
\leftsladjust[#3`#2`{#5}]
\rightsladjust[#4`#2`{#6}]
\begin{picture}(\xext,\yext)(\xoff,\yoff)%
\putAtrianglep<\arrowtypea`\arrowtypeb`\arrowtypec;\height>%
(0,0)[#2`#3`#4;#5`#6`{#7}]%
\end{picture}%
}}

\def\putAtrianglepairp<#1>(#2)[#3;#4`#5`#6`#7`#8]{{
\settripairparms[#1]%
\setpos(#2)%
\settokens[#3]%
\puthmorphism(\xpos,\ypos)[\tokenb`\tokenc`{#7}]{\height}{\arrowtyped}b%
\advance\xpos by\height
\advance\ypos by\height
\putmorphism(\xpos,\ypos)(-1,-1)[\tokena``{#4}]{\height}{\arrowtypea}l%
\putvmorphism(\xpos,\ypos)[``{#5}]{\height}{\arrowtypeb}m%
\putmorphism(\xpos,\ypos)(1,-1)[``{#6}]{\height}{\arrowtypec}r%
}}

\def\putAtrianglepair{\@ifnextchar <{\putAtrianglepairp}{\putAtrianglepairp%
   <\arrowtypea`\arrowtypeb`\arrowtypec`\arrowtyped`\arrowtypee;\height>}}
\def\Atrianglepair{\@ifnextchar <{\Atrianglepairp}{\Atrianglepairp%
   <\arrowtypea`\arrowtypeb`\arrowtypec`\arrowtyped`\arrowtypee;\height>}}

\def\Atrianglepairp<#1>[#2;#3`#4`#5`#6`#7]{{%
\settripairparms[#1]%
\settokens[#2]%
\width=\height
\xext=\width
\yext=\height
\topadjust[\tokena``]%
\vertadjust[\tokenb`\tokenc`{#6}]
\tempcountd=\tempcounta                       
\vertadjust[\tokenc`\tokend`{#7}]
\ifnum\tempcounta<\tempcountd                 
\tempcounta=\tempcountd\fi                    
\advance \yext by\tempcounta                  
\advance \yoff by-\tempcounta                 %
\multiply \xext by2 
\leftsladjust[\tokenb`\tokena`{#3}]
\rightsladjust[\tokend`\tokena`{#5}]%
\begin{picture}(\xext,\yext)(\xoff,\yoff)%
\putAtrianglepairp
<\arrowtypea`\arrowtypeb`\arrowtypec`\arrowtyped`\arrowtypee;\height>%
(0,0)[#2;#3`#4`#5`#6`{#7}]%
\end{picture}%
}}

\def\putVtrianglep<#1>(#2,#3)[#4`#5`#6;#7`#8`#9]{{%
\settriparms[#1]%
\xpos=#2 \ypos=#3
\advance\ypos by\height
{\multiply\height by2
\puthmorphism(\xpos,\ypos)[#4`#5`{#7}]{\height}{\arrowtypea}a}%
\putmorphism(\xpos,\ypos)(1,-1)[`#6`{#8}]{\height}{\arrowtypeb}l%
\advance\xpos by\height
\advance\xpos by\height
\putmorphism(\xpos,\ypos)(-1,-1)[``{#9}]{\height}{\arrowtypec}r%
}}

\def\putVtriangle{\@ifnextchar <{\putVtrianglep}{\putVtrianglep
   <\arrowtypea`\arrowtypeb`\arrowtypec;\height>}}
\def\Vtriangle{\@ifnextchar <{\Vtrianglep}{\Vtrianglep
   <\arrowtypea`\arrowtypeb`\arrowtypec;\height>}}

\def\Vtrianglep<#1>[#2`#3`#4;#5`#6`#7]{{
\settriparms[#1]
\width=\height                         
\xext=\width                           
\yext=\height                          
\topadjust[#2`#3`{#5}]
\botadjust[#4``]
\multiply \xext by2 
\leftsladjust[#2`#3`{#6}]
\rightsladjust[#3`#4`{#7}]
\begin{picture}(\xext,\yext)(\xoff,\yoff)%
\putVtrianglep<\arrowtypea`\arrowtypeb`\arrowtypec;\height>%
(0,0)[#2`#3`#4;#5`#6`{#7}]%
\end{picture}%
}}

\def\putVtrianglepairp<#1>(#2)[#3;#4`#5`#6`#7`#8]{{
\settripairparms[#1]%
\setpos(#2)%
\settokens[#3]%
\advance\ypos by\height
\putmorphism(\xpos,\ypos)(1,-1)[`\tokend`{#6}]{\height}{\arrowtypec}l%
\puthmorphism(\xpos,\ypos)[\tokena`\tokenb`{#4}]{\height}{\arrowtypea}a%
\advance\xpos by\height
\putvmorphism(\xpos,\ypos)[``{#7}]{\height}{\arrowtyped}m%
\advance\xpos by\height
\putmorphism(\xpos,\ypos)(-1,-1)[``{#8}]{\height}{\arrowtypee}r%
}}

\def\putVtrianglepair{\@ifnextchar <{\putVtrianglepairp}{\putVtrianglepairp%
    <\arrowtypea`\arrowtypeb`\arrowtypec`\arrowtyped`\arrowtypee;\height>}}
\def\Vtrianglepair{\@ifnextchar <{\Vtrianglepairp}{\Vtrianglepairp%
    <\arrowtypea`\arrowtypeb`\arrowtypec`\arrowtyped`\arrowtypee;\height>}}

\def\Vtrianglepairp<#1>[#2;#3`#4`#5`#6`#7]{{%
\settripairparms[#1]%
\settokens[#2]
\xext=\height                  
\width=\height                 
\yext=\height                  
\vertadjust[\tokena`\tokenb`{#4}]
\tempcountd=\tempcounta        
\vertadjust[\tokenb`\tokenc`{#5}]
\ifnum\tempcounta<\tempcountd%
\tempcounta=\tempcountd\fi
\advance \yext by\tempcounta
\botadjust[\tokend``]%
\multiply \xext by2
\leftsladjust[\tokena`\tokend`{#6}]%
\rightsladjust[\tokenc`\tokend`{#7}]%
\begin{picture}(\xext,\yext)(\xoff,\yoff)%
\putVtrianglepairp
<\arrowtypea`\arrowtypeb`\arrowtypec`\arrowtyped`\arrowtypee;\height>%
(0,0)[#2;#3`#4`#5`#6`{#7}]%
\end{picture}%
}}

\def\putCtrianglep<#1>(#2,#3)[#4`#5`#6;#7`#8`#9]{{%
\settriparms[#1]%
\xpos=#2 \ypos=#3
\advance\ypos by\height
\putmorphism(\xpos,\ypos)(1,-1)[``{#9}]{\height}{\arrowtypec}l%
\advance\xpos by\height
\advance\ypos by\height
\putmorphism(\xpos,\ypos)(-1,-1)[#4`#5`{#7}]{\height}{\arrowtypea}l%
{\multiply\height by 2
\putvmorphism(\xpos,\ypos)[`#6`{#8}]{\height}{\arrowtypeb}r}%
}}

\def\putCtriangle{\@ifnextchar <{\putCtrianglep}{\putCtrianglep
    <\arrowtypea`\arrowtypeb`\arrowtypec;\height>}}
\def\Ctriangle{\@ifnextchar <{\Ctrianglep}{\Ctrianglep
    <\arrowtypea`\arrowtypeb`\arrowtypec;\height>}}

\def\Ctrianglep<#1>[#2`#3`#4;#5`#6`#7]{{
\settriparms[#1]
\width=\height                          
\xext=\width                            
\yext=\height                           
\multiply \yext by2 
\topadjust[#2``]
\botadjust[#4``]
\sladjust[#3`#2`{#5}]{\width}
\tempcountd=\tempcounta                 
\sladjust[#3`#4`{#7}]{\width}
\ifnum \tempcounta<\tempcountd          
\tempcounta=\tempcountd\fi              
\advance \xext by\tempcounta            
\advance \xoff by-\tempcounta           %
\rightadjust[#2`#4`{#6}]
\begin{picture}(\xext,\yext)(\xoff,\yoff)%
\putCtrianglep<\arrowtypea`\arrowtypeb`\arrowtypec;\height>%
(0,0)[#2`#3`#4;#5`#6`{#7}]%
\end{picture}%
}}

\def\putDtrianglep<#1>(#2,#3)[#4`#5`#6;#7`#8`#9]{{%
\settriparms[#1]%
\xpos=#2 \ypos=#3
\advance\xpos by\height \advance\ypos by\height
\putmorphism(\xpos,\ypos)(-1,-1)[``{#9}]{\height}{\arrowtypec}r%
\advance\xpos by-\height \advance\ypos by\height
\putmorphism(\xpos,\ypos)(1,-1)[`#5`{#8}]{\height}{\arrowtypeb}r%
{\multiply\height by 2
\putvmorphism(\xpos,\ypos)[#4`#6`{#7}]{\height}{\arrowtypea}l}%
}}

\def\putDtriangle{\@ifnextchar <{\putDtrianglep}{\putDtrianglep
    <\arrowtypea`\arrowtypeb`\arrowtypec;\height>}}
\def\Dtriangle{\@ifnextchar <{\Dtrianglep}{\Dtrianglep
   <\arrowtypea`\arrowtypeb`\arrowtypec;\height>}}

\def\Dtrianglep<#1>[#2`#3`#4;#5`#6`#7]{{
\settriparms[#1]
\width=\height                         
\xext=\height                          
\yext=\height                          
\multiply \yext by2 
\topadjust[#2``]
\botadjust[#4``]
\leftadjust[#2`#4`{#5}]
\sladjust[#3`#2`{#5}]{\height}
\tempcountd=\tempcountd                
\sladjust[#3`#4`{#7}]{\height}
\ifnum \tempcounta<\tempcountd         
\tempcounta=\tempcountd\fi             
\advance \xext by\tempcounta           %
\begin{picture}(\xext,\yext)(\xoff,\yoff)
\putDtrianglep<\arrowtypea`\arrowtypeb`\arrowtypec;\height>%
(0,0)[#2`#3`#4;#5`#6`{#7}]%
\end{picture}%
}}

\def\setrecparms[#1`#2]{\width=#1 \height=#2}%
%

\def\recursep<#1`#2>[#3;#4`#5`#6`#7`#8]{{%
\width=#1 \height=#2
\settokens[#3]
\settowidth{\tempdimen}{$\tokena$}
\ifdim\tempdimen=0pt
  \savebox{\tempboxa}{\hbox{$\tokenb$}}%
  \savebox{\tempboxb}{\hbox{$\tokend$}}%
  \savebox{\tempboxc}{\hbox{$#6$}}%
\else
  \savebox{\tempboxa}{\hbox{$\hbox{$\tokena$}\times\hbox{$\tokenb$}$}}%
  \savebox{\tempboxb}{\hbox{$\hbox{$\tokena$}\times\hbox{$\tokend$}$}}%
  \savebox{\tempboxc}{\hbox{$\hbox{$\tokena$}\times\hbox{$#6$}$}}%
\fi
\ypos=\height
\divide\ypos by 2
\xpos=\ypos
\advance\xpos by \width
\xext=\xpos \yext=\height
\topadjust[#3`\usebox{\tempboxa}`{#4}]%
\botadjust[#5`\usebox{\tempboxb}`{#8}]%
\sladjust[\tokenc`\tokenb`{#5}]{\ypos}%
\tempcountd=\tempcounta
\sladjust[\tokenc`\tokend`{#5}]{\ypos}%
\ifnum \tempcounta<\tempcountd
\tempcounta=\tempcountd\fi
\advance \xext by\tempcounta
\advance \xoff by-\tempcounta
\rightadjust[\usebox{\tempboxa}`\usebox{\tempboxb}`\usebox{\tempboxc}]%
\bfig
\putCtrianglep<-1`1`1;\ypos>(0,0)[`\tokenc`;#5`#6`{#7}]%
\puthmorphism(\ypos,0)[\tokend`\usebox{\tempboxb}`{#8}]{\width}{-1}b%
\puthmorphism(\ypos,\height)[\tokenb`\usebox{\tempboxa}`{#4}]{\width}{-1}a%
\advance\ypos by \width
\putvmorphism(\ypos,\height)[``\usebox{\tempboxc}]{\height}1r%
\efig
}}

\def\recurse{\@ifnextchar <{\recursep}{\recursep<\width`\height>}}

\def\puttwohmorphisms(#1,#2)[#3`#4;#5`#6]#7#8#9{{%
%
\puthmorphism(#1,#2)[#3`#4`]{#7}0a
\ypos=#2
\advance\ypos by 20
\puthmorphism(#1,\ypos)[\phantom{#3}`\phantom{#4}`#5]{#7}{#8}a
\advance\ypos by -40
\puthmorphism(#1,\ypos)[\phantom{#3}`\phantom{#4}`#6]{#7}{#9}b
}}

\def\puttwovmorphisms(#1,#2)[#3`#4;#5`#6]#7#8#9{{%
%
%
%
\putvmorphism(#1,#2)[#3`#4`]{#7}0a
\xpos=#1
\advance\xpos by -20
\putvmorphism(\xpos,#2)[\phantom{#3}`\phantom{#4}`#5]{#7}{#8}l
\advance\xpos by 40
\putvmorphism(\xpos,#2)[\phantom{#3}`\phantom{#4}`#6]{#7}{#9}r
}}

\def\puthcoequalizer(#1)[#2`#3`#4;#5`#6`#7]#8#9{{%
%
\setpos(#1)%
\puttwohmorphisms(\xpos,\ypos)[#2`#3;#5`#6]{#8}11%
\advance\xpos by #8
\puthmorphism(\xpos,\ypos)[\phantom{#3}`#4`#7]{#8}1{#9}
}}

\def\putvcoequalizer(#1)[#2`#3`#4;#5`#6`#7]#8#9{{%
%
%
%
%
\setpos(#1)%
\puttwovmorphisms(\xpos,\ypos)[#2`#3;#5`#6]{#8}11%
\advance\ypos by -#8
\putvmorphism(\xpos,\ypos)[\phantom{#3}`#4`#7]{#8}1{#9}
}}

\def\putthreehmorphisms(#1)[#2`#3;#4`#5`#6]#7(#8)#9{{%
\setpos(#1) \settypes(#8)
\if a#9 %
     \vertsize{\tempcounta}{#5}%
     \vertsize{\tempcountb}{#6}%
     \ifnum \tempcounta<\tempcountb \tempcounta=\tempcountb \fi
\else
     \vertsize{\tempcounta}{#4}%
     \vertsize{\tempcountb}{#5}%
     \ifnum \tempcounta<\tempcountb \tempcounta=\tempcountb \fi
\fi
\advance \tempcounta by 60
\puthmorphism(\xpos,\ypos)[#2`#3`#5]{#7}{\arrowtypeb}{#9}
\advance\ypos by \tempcounta
\puthmorphism(\xpos,\ypos)[\phantom{#2}`\phantom{#3}`#4]{#7}{\arrowtypea}{#9}
\advance\ypos by -\tempcounta \advance\ypos by -\tempcounta
\puthmorphism(\xpos,\ypos)[\phantom{#2}`\phantom{#3}`#6]{#7}{\arrowtypec}{#9}
}}

\def\putarc(#1,#2)[#3`#4`#5]#6#7#8{{%
\xpos #1
\ypos #2
\width #6
\arrowlength #6
\putbox(\xpos,\ypos){#3\vphantom{#4}}%
{\advance \xpos by\arrowlength
\putbox(\xpos,\ypos){\vphantom{#3}#4}}%
\horsize{\tempcounta}{#3}%
\horsize{\tempcountb}{#4}%
\divide \tempcounta by2
\divide \tempcountb by2
\advance \tempcounta by30
\advance \tempcountb by30
\advance \xpos by\tempcounta
\advance \arrowlength by-\tempcounta
\advance \arrowlength by-\tempcountb
\halflength=\arrowlength \divide\halflength by 2
\divide\arrowlength by 5
\put(\xpos,\ypos){\bezier{\arrowlength}(0,0)(50,50)(\halflength,50)}
\ifnum #7=-1 \put(\xpos,\ypos){\vector(-3,-2)0} \fi
\advance\xpos by \halflength
\put(\xpos,\ypos){\xpos=\halflength \advance\xpos by -50
   \bezier{\arrowlength}(0,50)(\xpos,50)(\halflength,0)}
\ifnum #7=1 {\advance \xpos by
   \halflength \put(\xpos,\ypos){\vector(3,-2)0}} \fi
\advance\ypos by 50
\vertsize{\tempcounta}{#5}%
\divide\tempcounta by2
\advance \tempcounta by20
\if a#8 %
   \advance \ypos by\tempcounta
   \putbox(\xpos,\ypos){#5}%
\else
   \advance \ypos by-\tempcounta
   \putbox(\xpos,\ypos){#5}%
\fi
}}

\makeatother

\sloppy

\def\lk{\langle}
\def\rk{\rangle}

\newcommand{\pa}{\parallel}
\newcommand{\lra}{\longrightarrow}
\newcommand{\hra}{\hookrightarrow}
\newcommand{\hla}{\hookleftarrow}
\newcommand{\ra}{\rightarrow}
\newcommand{\la}{\leftarrow}
\newcommand{\lla}{\longleftarrow}
\newcommand{\da}{\downarrow}
\newcommand{\ua}{\uparrow}
\newcommand{\Da}{\Downarrow}
\newcommand{\Ua}{\Uparrow}
\newcommand{\Lra}{\Longrightarrow}
\newcommand{\Ra}{\Rightarrow}
\newcommand{\Lla}{\Longleftarrow}
\newcommand{\La}{\Leftarrow}

\newcommand{\lms}{\longmapsto}
\newcommand{\ms}{\mapsto}

\def\o{{\omega}}

\def\bA{{\bf A}}
\def\bM{{\bf M}}
\def\bN{{\bf N}}
\def\bC{{\bf C}}
\def\bI{{\bf I}}
\def\bL{{\bf L}}
\def\bT{{\bf T}}
\def\bS{{\bf S}}
\def\bD{{\bf D}}
\def\bB{{\bf B}}
\def\bW{{\bf W}}
\def\bP{{\bf P}}
\def\bX{{\bf X}}
\def\bY{{\bf Y}}

\def\ba{{\bf a}}
\def\bb{{\bf b}}
\def\bc{{\bf c}}
\def\bd{{\bf d}}
\def\bh{{\bf h}}
\def\bi{{\bf i}}
\def\bj{{\bf j}}
\def\bk{{\bf k}}
\def\bm{{\bf m}}
\def\bn{{\bf n}}
\def\bp{{\bf p}}
\def\bq{{\bf q}}
\def\be{{\bf e}}
\def\br{{\bf r}}
\def\bi{{\bf i}}
\def\bs{{\bf s}}
\def\bt{{\bf t}}

\def\cBL{{\cal BL}}
\def\cB{{\cal B}}
\def\cA{{\cal A}}
\def\cC{{\cal C}}
\def\cD{{\cal D}}
\def\cE{{\cal E}}
\def\cF{{\cal F}}
\def\cG{{\cal G}}
\def\cI{{\cal I}}
\def\cJ{{\cal J}}
\def\cK{{\cal K}}
\def\cL{{\cal L}}
\def\cN{{\cal N}}
\def\cM{{\cal M}}
\def\cO{{\cal O}}
\def\cP{{\cal P}}
\def\cQ{{\cal Q}}
\def\cR{{\cal R}}
\def\cS{{\cal S}}
\def\cT{{\cal T}}
\def\cU{{\cal U}}
\def\cV{{\cal V}}
\def\cW{{\cal W}}
\def\cX{{\cal X}}
\def\cY{{\cal Y}}


\def\Mnd{{\bf Mnd}}
\def\bDel{{{\bf \Delta}}}
\def\bCat{{{\bf Cat}}}

\pagenumbering{arabic} \setcounter{page}{1}

\title{\bf\Large The Partial Simplicial Category and Algebras for Monads}

\author{Marek Zawadowski\\
Instytut Matematyki, Uniwersytet Warszawski\\
ul. S.Banacha 2,\\
00-913 Warszawa, Poland\\
zawado@mimuw.edu.pl\\
}

\maketitle
\begin{abstract}
We construct explicitly the weights on the simplicial category so that
the colimits and limits of 2-functors with those weights
provide the Kleisli objects and  the Eilenberg-Moore objects,
respectively, in any 2-category.

MS Classification 18A30, 18A40 (AMS 2010).
\end{abstract}

\section{Introduction}
It is well know that monads in $2$-categories correspond to $2$-functors from the simplicial category $\Delta$.
It is also well known that the Kleisli and the Eilenberg-Moore objects can be build as weighted (co)limits, c.f. \cite{St2}. In this paper
we construct explicitly the weights $W_r$ and $W_l$ on $\Delta$ so that the $W_r$-weighted colimits provide the Kleisli objects
and the $W_l$-weighted limits provide  the Eilenberg-Moore objects in any $2$-category $\cD$.

\section{Partial simplicial category $\Pi$}
Let $\Delta$ be the usual (algebraists) simplicial category. The objects of $\Delta$ are finite linear orders
denoted by ${n}=\lk n,\leq \rk=\lk \{0,\ldots,n-1 \},\leq \rk$, for $n\in \o$. The morphisms of
$\Delta$ are monotone functions. For $n\geq 1$ and $0\leq i < n$, the morphism
\[ \sigma^n_i : {n+1} \lra {n}\]
is an epi that takes value $i$ twice. For $n\geq 0$ and $0\leq i \leq n$, the morphism
\[ \delta^n_i : {n} \lra {n+1}\]
is a mono that misses the value $i$. We usually omit the upper index when it can be read from the context.
These morphisms satisfy the following simplicial identities. For $i\leq j$
\[ \delta_i\delta_j =\delta_{j+1}\delta_i\hskip 2cm \sigma_j\sigma_i=\sigma_i\sigma_{j+1} \]
and \[ \sigma_j\delta_i \;= \; \left\{ \begin{array}{ll}
                \delta_i\sigma_{j-1}    & \mbox{ if  $i<j$} \\
                1     & \mbox{ if  $i=j,j+1$} \\
                \delta_{i-1}\sigma_j    & \mbox{ if  $i> j+1$}
                                    \end{array}
                            \right.\label{sigmadelta} \]
It is well known, c.f \cite{CWM}, that the morphisms $\sigma_i$ and $\delta_i$ generates $\Delta$
subject to the above relations.

The partial simplicial category $\Pi$ has the same objects as $\Delta$ but the morphisms
in $\Pi$ are partial monotone functions. Clearly, the morphism $\sigma_i$ and $\delta_i$ of $\Delta$ are morphism in $\Pi$ as well,
and they satisfy the same simplicial identities. For $n\geq 0$ and $0\leq i \leq n$, the morphism
\[ \tau^n_i : {n+1} \lra {n}\]
is an epi not defined on $i$. For $i\leq j$ we have
\[ \tau_j\tau_i=\tau_i\tau_{j+1} \]
Moreover, the morphisms $\sigma_i$, $\delta_i$, and $\tau_i$ satisfy the following identities in $\Pi$ that,
together with the above identities, will be called {\em partial simplicial identities}.
 \[ \tau_j\sigma_i \;= \; \left\{ \begin{array}{ll}
                \sigma_i\tau_{j+1}    & \mbox{ if  $i<j$} \\
                \tau_j\tau_{j+1}    & \mbox{ if  $i=j$} \\
                \sigma_{i+1}\tau_j    & \mbox{ if  $i> j$}
                                    \end{array}
                            \right.\label{sigmatau} \hskip 1cm
                            \tau_j\delta_i \;= \; \left\{ \begin{array}{ll}
                \delta_i\tau_{j+1}    & \mbox{ if  $i<j$} \\
                1    & \mbox{ if  $i=j$} \\
                \delta_{i-1}\tau_j    & \mbox{ if  $i> j$}
                                    \end{array}
                            \right.\label{deltatau}  \]
We have
\begin{lemma}\label{nf}
Every morphism $f:n\ra m$ in $\Pi$ can be expressed in a canonical form as
\begin{equation} f=\delta_{i_r}\ldots\delta_{i_1}\sigma_{j_1}\ldots\sigma_{j_s}\tau_{k_1}\ldots \tau_{k_t}
\end{equation}
with $i_1<\ldots <i_r$, $j_1<\ldots <j_s$, and $k_1<\ldots < k_t$, $m-n=r-s-t$.
\end{lemma}

\begin{proof} This can be easily seen directly or using the partial simplicial identities.
\end{proof}

\vskip 2mm
\begin{theorem}
The category $\Pi$ is generated by the morphisms $\sigma_i$, $\delta_i$, and $\tau_i$ subject to the partial simplicial identities.
\end{theorem}

\begin{proof} The partial simplicial identities hold in $\Pi$. Moreover, every morphism in $\Pi$
can be written in a canonical form. Finally, two different canonical forms represent two different morphisms in $\Pi$.
\end{proof}
\vskip 2mm

{\em Remark.} The category $\Pi$ is a strict monoidal category with the monoidal structure defined by the coproduct. Moreover, the inclusion functor
$\Delta \ra \Pi$ is a strict morphism of strict monoidal categories.

\section{The categories $\Pi_l$ and $\Pi_r$ and the multiplication functors}
As the category $\Delta$ is a strict monoidal category it can be considered
as $2$-category with one $0$-cell $\ast$, and then the tensor becomes the composition of $1$-cells.
We denote this $2$-category by  $\bDel$.

The {\em left partial simplicial category} $\Pi_l$ is a subcategory of $\Pi$ with the same objects as $\Pi$. A morphism $f:n\ra m$ from $\Pi$ is in $\Pi_l$ iff
for any $i\leq j \in n$, if $f(j)$ is defined so is $f(i)$. In other words the morphisms of $\Pi$ are generated by the morphisms in $\Delta$ and the morphism $\tau^n_n:n+1\ra n$ for $n\in \o$. Note that $\tau^n_n=id_n+\tau^0_0$. In $\Pi_l$  Lemma \ref{nf} holds with an additional condition that $k_{i+1}=k_i+1$ and $k_t=n-1$.

We have a {\em left multiplication} $2$-functor \[ W_l:\bDel \ra 2Cat \]
such that
\[ W_l(\ast)=\Pi_l,\;\;\; W_l(n) = n+(-),\;\;\; W_l(f) = f+(-) \]
for $f:n\ra m\in \Delta$. Thus $W_l$ can be seen as an action od $\Delta$ on $\Pi_l$ by tensoring on the left.
 Clearly, $\Pi_l$ is closed with respect to such operations.

Dually, we have the {\em right partial simplicial category} $\Pi_r$, a subcategory of $\Pi$
with the same objects as $\Pi$. A morphism $f:n\ra m$ from $\Pi$ is in $\Pi_r$ iff
for any $i\leq j \in n$, if $f(i)$ is defined so is $f(j)$. In other words the morphisms of $\Pi$ are generated by the morphisms in $\Delta$ and the morphism $\tau^n_0:n+1\ra n$, for $n\in \o$. Note that $\tau^n_0=\tau^0_0+id_n$. In $\Pi_l$  Lemma \ref{nf} holds with a additional condition that $k_{i+1}=k_i+1$ and $k_1=0$.

We have a {\em right multiplication} $2$-functor \[ W_r:\bDel \ra 2Cat \]
such that
\[ W_r(\ast)=\Pi_r,\;\;\; W_r(n) = (-)+n,\;\;\; W_r(f) = (-)+f \]
for $f:n\ra m\in \Delta$. Thus $W_r$ can be seen as an action od $\Delta$ on $\Pi_r$ by tensoring on the right.
Clearly, $\Pi_r$ is closed with respect to such operations.

\section{Monads as $2$-functors}

The $2$-functors $\bT : \bDel \ra \cD$ correspond bijectively to monads in the $2$-category $\cD$.

Suppose $(\cC,T,\eta,\mu)$ is a monad in $\cD$ on $\cC$. We define a $2$-functor $\bT$ as follows:
\[ \bT(\ast)=\cC,\;\;\;     \bT(0)=1_\cC,\;\;\;\bT(n)=T^n,\]
\[\bT(\delta^0_0)=\eta,\;\;\; \bT(\delta^n_i)=T^{n-i}\eta_{T^i},\;\;\; \bT(\sigma^1_0)=\mu,\;\;\; \bT(\sigma^n_i)=T^{n-i-1}\mu_{T^i},\;\;\; \]
The equations
\[ \bT(\sigma_i)\circ \bT(\sigma_i) =  \bT(\sigma_i)\circ \bT(\sigma_{i+1}) \]
hold, as the consequence of the associativity of the multiplication $\mu\circ T\mu = \mu\circ \mu_T$.
The equations
\[ \bT(\sigma_i)\circ \bT(\delta_i) = 1= \bT(\sigma_{i+1})\circ \bT(\delta_i) \]
hold, as the consequence of the unit axiom $\mu\circ T\eta = 1 = \mu\circ \eta_T$.
The remaining simplicial equations hold as a consequence of the Middle Exchange Law (MEL).

On the other hand, having a $2$-functor $\bT : \bDel \ra \cD$,  we get a monad $(\bT(\ast),  \bT(1),  \bT(\delta^0_0), \bT(\sigma^1_0))$.

Let $2\bCat$ be the $3$-category of $2$-categories. As $2\bCat(\bDel, \cD)$ is the $2$-category of monads in $\cD$ with strict morphisms,
we can think of $\bDel$ as a $2$-category representing monads with strict morphisms in $2$-categories.

\section{The $2$-functor $Subeq_T$ and the Eilenberg-Moore objects}

For a given monad $(\cC,T,\eta,\mu)$ in a $2$-category $\cD$ we define a $2$-functor
\[ Subeq_T : \cD^{op} \lra \bCat\]
as follows. For a given $0$-cell $X$ in $\cD$, the category $Subeq_T(X)$ has as objects pairs $(U,\xi)$
such that $U:X\ra \cC$ is a $1$-cell in $\cD$, $\xi:TU\ra U$ is a $2$-cell in $\cD$ such that in the diagram
\begin{center} \xext=1600 \yext=280
\begin{picture}(\xext,\yext)(\xoff,\yoff)
   \putmorphism(0,100)(1,0)[T^2 U`T U`]{800}{0}a
   \putmorphism(0,50)(1,0)[\phantom{T^2 U}`\phantom{T U}`\mu_U]{800}{1}b
   \putmorphism(0,150)(1,0)[\phantom{T^2 U}`\phantom{T U}`T(\xi)]{800}{1}a

  \putmorphism(800,100)(1,0)[\phantom{TU}`U`]{800}{0}a
  \putmorphism(800,150)(1,0)[\phantom{TU}`\phantom{U}`\xi]{800}{1}a
  \putmorphism(800,50)(1,0)[\phantom{T U}`\phantom{U}`\eta_U]{800}{-1}b
\end{picture}
\end{center}
we have
\[ \xi\circ \eta_U = 1_U,\hskip 1cm \xi\circ T(\xi) = \xi \circ \mu_U.\]
In such case, we say that $(U,\xi)$ {\em subequalizes the monad} $T$.
A morphism $\tau :(U,\xi)\ra (U',\xi')$ is a $2$-cell $\tau: U\ra U'$ such that
the square
\begin{center} \xext=500 \yext=450
\begin{picture}(\xext,\yext)(\xoff,\yoff)
\setsqparms[1`1`1`1;600`400]
 \putsquare(0,0)[TU`TU'`U`U';T(\tau)`\xi`\xi'`\tau]
 \end{picture}
\end{center}
commutes. The $2$-functor $Subeq_T$ is define on $1$- and $2$-cells in the obvious way, by composition.

Recall, from \cite{St1}, see \cite{Z} for the notation, that  {\em the $2$-category $\cD$
admits Eilenberg-Moore objects} if the embedding $\iota$
\begin{center} \xext=1000 \yext=300
\begin{picture}(\xext,\yext)(\xoff,\yoff)
\putmorphism(0,150)(1,0)[\cD`\Mnd(\cD)`\iota]{1000}{1}a

 \putmorphism(0,50)(1,0)[\phantom{\cD}`\phantom{\Mnd(\cD)}`{EM}]{1000}{-1}b
 \end{picture}
\end{center}
has a right $2$-adjoint $\iota\dashv EM$. $\Mnd(\cD)$ is the $2$-category
of monads in $\cD$ with lax morphisms of monads and transformations of lax morphisms.
We have a $2$-functor
\begin{equation}\label{iota} \Mnd(\iota(-),T) : \cD^{op} \lra \bCat\end{equation}
sending $0$-cell $X$ in $\cD$ to the category  $\Mnd(\iota(X),T)$ of lax morphisms
from the identity monad on $X$ to the monad $T$ and transformations between such morphisms.

The following definition is a 'monad by monad' version of the previous one.
We say that {\em  the monad $T$ admits Eilenberg-Moore object}
iff the $2$-functor $\Mnd(\iota(-),T)$  is representable.

A simple verification shows

\begin{lemma}\label{subeq-mnd}
The $2$-functors $Subeq_T$ and  $\Mnd(\iota(-),T)$ are naturally isomorphic.
\end{lemma}

\section{The $2$-functor $Cone_{W_l}(\bT)$}
Let $\bT:\Delta\ra \cD$ be a $2$-functor and $T$ be the monad corresponding to $\bT$. We define a  $2$-functor
\[ Cone_{W_l}(\bT) : \cD^{op} \lra \bCat\]
of $W_l$-cones over $\bT$. We will show that the $2$-functor $Cone_{W_l}(\bT)$ is isomorphic to  the $2$-functor $Subeq_T$.

Fix a $0$-cell $X$ in $\cD$.
An object of $Cone_{W_l}(\bT)(X)$ is a $2$-natural transformation $$\lambda: W_l\lra \cD(X,\bT(\ast))$$
with only one component the functor $\lambda_\ast$ also denoted $\lambda$. The $2$-naturality means that
the square
\begin{center} \xext=1600 \yext=750
\begin{picture}(\xext,\yext)(\xoff,\yoff)
\setsqparms[1`0`0`1;1600`600]
 \putsquare(150,50)[\Pi_l`\cD(X,\bT(\ast))`\Pi_l`\cD(X,\bT(\ast));\lambda```\lambda]
 \putmorphism(0,650)(0,-1)[\phantom{\Pi_l}`\phantom{\Pi_l}`_{m+(-)}]{600}{1}l
  \putmorphism(300,650)(0,-1)[\phantom{\Pi_l}`\phantom{\Pi_l}`_{n+(-)}]{600}{1}r
    \put(40,305){$_{f+(-)}$}
    \put(80,205){$\Longrightarrow$}
    \putmorphism(1550,650)(0,-1)[\phantom{\cD(X,\bT(\ast))}`\phantom{\cD(X,\bT(\ast))}`_{\cD(X,\bT(m))}]{600}{1}l
  \putmorphism(1950,650)(0,-1)[\phantom{\cD(X,\bT(\ast))}`\phantom{\cD(X,\bT(\ast))}`_{\cD(X,\bT(n))}]{600}{1}r
    \put(1570,305){$_{\cD(X,\bT(f))}$}
    \put(1680,205){$\Longrightarrow$}
  \putmorphism(0,100)(0,-1)[\phantom{}`\phantom{}`]{800}{0}a
 \end{picture}
\end{center}
commutes, for any $f: m \ra n$ in $\Delta$. Put $T=\bT(1)$ and $\lambda(0)=U:X \ra\bT(\ast)$. We have with $f$ as before
\[  \lambda(m) =\lambda(m+0) = \bT(m)\circ \lambda(0) =  T^mU,\]
\[ \lambda(f) =\lambda(f+0) = \bT(f)\circ \lambda(0) =  \bT(f)\circ U\]
Moreover, putting $\lambda(\tau^0_0)=\xi:TU\ra U$, we have
\[ \lambda(\tau^n_n) = \lambda(id_n+\tau^0_0) = \bT(id_n)\circ_0 \bT(\tau^0_0)= \bT(id_n)\circ_0 \xi =\]
\[=\bT(id_1)\circ_0  \ldots\circ_0 \bT(id_1)\circ_0 \xi=id_T\circ_0  \ldots\circ_0 id_T\circ_0 \xi=  T^n(\xi) \]
Thus $\lambda$ is uniquely determined by $U$ and $\xi$. The equations
\[ \tau^0_0\circ \delta^1_0=1  ,\hskip 5mm \tau^0_0\circ \tau^1_1=\tau^0_0\circ \sigma^1_0 \]
implies that $(U,\xi)$ subequalizes the monad $T$.
On the other hand, if $(U,\xi)$ subequalizes $T$ then we can define a $2$-natural transformation
$\lambda:W_l\ra \cD(X,\bT)$, as follows. The functor
\[ \lambda=\lambda_\ast : \Pi_l \lra \cD(X,\Pi(\ast)) \]
is defined so that
\[ \lambda(0)=U,\hskip 2mm \lambda(\tau^0_0)= \xi\]
and for $2$-naturality of $\lambda$ we have
\[  \lambda(n) =T^nU,\hskip 2mm \lambda(f)= \bT(f)_U,\hskip 2mm  \lambda(\tau^n_n)= T^n(\xi) \]
Then it is easy to verify that $\lambda$ respect all the equations in $\Pi_l$.

A morphism in $Cone_{W_l}(\bT)(X)$ between two $2$-natural transformations is a modification
$\nu: \lambda \ra \lambda'$ with one component $\nu_\ast$,
denoted also $\nu$, that is a natural transformation so that, for any $n\in\o$, the square
\begin{center} \xext=1750 \yext=950
\begin{picture}(\xext,\yext)(\xoff,\yoff)
\setsqparms[0`1`1`0;1600`700]
 \putsquare(150,120)[\Pi_l`\cD(X,\bT(\ast))`\Pi_l`\cD(X,\bT(\ast));`_{n+(-)}`_{\cD(X,\bT(n))}`]
  \putmorphism(150,200)(1,0)[\phantom{\Pi_l}`\phantom{\cD(X,\bT(\ast))}`\lambda]{1600}{1}a
  \putmorphism(150,50)(1,0)[\phantom{\Pi_l}`\phantom{\cD(X,\bT(\ast))}`\lambda']{1600}{1}b
  \put(1100,120){$_{\nu\,\Downarrow}$}

 \putmorphism(150,900)(1,0)[\phantom{\Pi_l}`\phantom{\cD(X,\bT(\ast))}`\lambda]{1600}{1}a
 \putmorphism(150,750)(1,0)[\phantom{\Pi_l}`\phantom{\cD(X,\bT(\ast))}`\lambda']{1600}{1}b
 \put(1100,820){$_{\nu\,\Downarrow}$}

 \end{picture}
\end{center}
commutes. As we have
\[ \nu_n = \nu_{(n+0)} = T^n(\nu_0) \]
the modification $\nu$ is uniquely determined by $\nu_0: U\ra U'=\lambda'(0)$. The square
\begin{center} \xext=500 \yext=520
\begin{picture}(\xext,\yext)(\xoff,\yoff)
\setsqparms[1`1`1`1;600`400]
 \putsquare(0,40)[\lambda(1)`\lambda'(1)`\lambda(0)`\lambda'(0);\nu_1`\lambda(\tau^0_0)`\lambda'(\tau^0_0)`\nu_0]
 \end{picture}
\end{center}
commutes as it is the naturality of $\nu_\ast : \lambda_\ast\ra \lambda'_\ast$ on $\tau^0_0$.

On the other hand, any $2$-cell $\nu_0: U\ra U'$ in $\cD$ such that
\begin{center} \xext=500 \yext=540
\begin{picture}(\xext,\yext)(\xoff,\yoff)
\setsqparms[1`1`1`1;600`400]
 \putsquare(0,40)[TU`TU'`U`U';T(\nu_0)`\xi`\xi'`\nu_0]
 \end{picture}
\end{center}
extends to a natural transformation from $\lambda_\ast$ to $\lambda'_\ast$, i.e.
a modification $\nu$ from $\lambda$ to $\lambda'$.

The $2$-functor  $Cone_{W_l}(\bT)$ is defined on $1$- and $2$-cells in the obvious way.

Constructing this functor we have in fact proved

\begin{lemma}\label{subeq-cone}
The $2$-functors $Subeq_T$ and  $Cone_{W_l}(\bT)$ are naturally isomorphic.
\end{lemma}

\section{The Eilenberg-Moore objects}
\begin{theorem}
Let $(\cC,T,\eta,\mu)$ be a monad in a $2$-category $\cD$ and $\bT:\bDel \ra \cD$ the corresponding $2$-functor. Then
$T$ admits Eilenberg-Moore object iff $\bT$ has a $W_l$-weighted limit. If it is the case then the Eilenberg-Moore object for $T$
and  the $W_l$-weighted limit of $\bT$ are isomorphic.
\end{theorem}

\begin{proof} By Lemmas \ref{subeq-mnd} and \ref{subeq-cone},
the $2$-functors $\Mnd(\iota(-),T)$ and  $Cone_{W_l}(\bT)$ are naturally isomorphic.
So if one is of representable so is the other and the representing objects are isomorphic.
The representation of the first give rise to the Eilenberg-Moore object for $T$,
and the representation of the second give rise to the $W_l$-weighted limit of $\bT$.
\end{proof}

From the above theorem we get immediately

\begin{corollary}
Any $2$-category $\cD$ admits Eilenberg-Moore object iff it has all $W_l$-weighted limit
of $2$-functors from $\Delta$.
\end{corollary}

\section{The Kleisli objects}
Clearly, all the above considerations can be dualised. In this case we get results
relating Kleisli objects and the $W_r$-weighted colimits of $2$-functors from $\Delta$.

We note for the record

\begin{theorem}
Let $(\cC,T,\eta,\mu)$ be a monad in a $2$-category $\cD$ and $\bT:\Delta \ra \cD$ the corresponding $2$-functor. Then
$T$ admits Kleisli object iff $\bT$ has a $W_r$-weighted colimit. If it is the case then the Kleisli object for $T$
and  the $W_r$-weighted colimit of $\bT$ are isomorphic.
\end{theorem}

\begin{corollary}
Any $2$-category $\cD$ admits Kleisli object iff it has all $W_r$-weighted colimit
of $2$-functors from $\Delta$.
\end{corollary}

\section{Appendix: Weighted limits in $2$-categories}

We recall the definition of weighted limits in $2$-categories in detail.

\subsection*{The $2$-functor $\cD(X,\bT)$}

For two $2$-functors between $2$-categories as shown\footnote{There are some foundational problems that one should address.
For example, it is desirable that the $2$-category $\cI$ be small. But we will be ignoring this
issues believing that the reader can fix all these problem on its own, the way she or he likes most.}
\begin{center} \xext=1500 \yext=250
\begin{picture}(\xext,\yext)(\xoff,\yoff)
  \putmorphism(0,0)(1,0)[\cI`Cat`W]{500}{1}a
  \putmorphism(1000,0)(1,0)[\cI`\cD`\bT]{500}{1}a
 \end{picture}
\end{center}
we are going to describe the $W$-weighted limit of $\bT$.

For any $0$-cell $X$ in $\cD$ we can form a $2$-functor
\begin{center} \xext=1800 \yext=1050
\begin{picture}(\xext,\yext)(\xoff,\yoff)
\setsqparms[0`0`0`0;1600`600]
\putmorphism(150,900)(1,0)[ \cI`\bCat`\cD(X,\bT)]{1550}{1}a
 \putsquare(150,50)[i`\cD(X,\bT(i))`j`\cD(X,\bT(j));```]
 \putmorphism(0,650)(0,-1)[\phantom{i}`\phantom{j}`f]{600}{1}l
  \putmorphism(300,650)(0,-1)[\phantom{i}`\phantom{j}`g]{600}{1}r
    \put(130,305){$\alpha$}
    \put(80,205){$\Longrightarrow$}
    \putmorphism(1550,650)(0,-1)[\phantom{\cD(X,\bT(i))}`\phantom{\cD(X,\bT(j))}`_{\cD(X,\bT(f))}]{600}{1}l
  \putmorphism(1950,650)(0,-1)[\phantom{\cD(X,\bT(i))}`\phantom{\cD(X,\bT(j))}`_{\cD(X,\bT(f))}]{600}{1}r
    \put(1570,305){$_{\cD(X,\bT(\alpha))}$}
    \put(1680,205){$\Longrightarrow$}

  \put(550,350){\line(0,1){100}}
  \put(550,400){\vector(1,0){400}}
 \end{picture}
\end{center}
of 'homming into' $\bT$.

The category $\cD(X,\bT(i))$ consists of $1$- and $2$-cells in $\cD$ from $X$ to $\bT(i)$.

The functor
\begin{center} \xext=1200 \yext=150
\begin{picture}(\xext,\yext)(\xoff,\yoff)
  \putmorphism(0,0)(1,0)[\cD(X,\bT(i))`\cD(X,\bT(j))`\cD(X,\bT(f))]{1200}{1}a
\end{picture}
\end{center}
is a whiskering along the $2$-cell $\bT(f)$:
\begin{center} \xext=2500 \yext=200
\begin{picture}(\xext,\yext)(\xoff,\yoff)
  \putmorphism(0,60)(1,0)[X `\bT(i)`]{800}{0}a
    \putmorphism(0,150)(1,0)[\phantom{X} `\phantom{\bT(i)}`r]{800}{1}a
    \putmorphism(0,0)(1,0)[\phantom{X} `\phantom{\bT(i)}`s]{800}{1}b
     \put(400,50){$\gamma\Da$}

       \put(1000,30){\line(0,1){100}}
  \put(1000,80){\vector(1,0){400}}

      \putmorphism(1500,60)(1,0)[X `\bT(j)`]{1000}{0}a
    \putmorphism(1500,150)(1,0)[\phantom{X} `\phantom{\bT(j)}`\bT(f)\circ r]{1000}{1}a
    \putmorphism(1500,0)(1,0)[\phantom{X} `\phantom{\bT(j)}`\bT(f)\circ s]{1000}{1}b
     \put(1700,50){$\bT(f)(\gamma)\Da$}
\end{picture}
\end{center}

The component of the natural transformation
\begin{center} \xext=1400 \yext=150
\begin{picture}(\xext,\yext)(\xoff,\yoff)
  \putmorphism(0,0)(1,0)[\cD(X,\bT(f))`\cD(X,\bT(g))`\cD(X,\bT(\alpha))]{1400}{1}a
\end{picture}
\end{center}
at $r:X\ra \bT(i)$ is
\begin{center} \xext=1200 \yext=150
\begin{picture}(\xext,\yext)(\xoff,\yoff)
  \putmorphism(0,0)(1,0)[\bT(f)\circ r`\bT(g)\circ r`\bT(\alpha)_r]{1200}{1}a
\end{picture}
\end{center}
The naturality of $\cD(X,\bT(f))$
\begin{center} \xext=800 \yext=600
\begin{picture}(\xext,\yext)(\xoff,\yoff)
\setsqparms[1`1`1`1;800`450]
 \putsquare(0,50)[\bT(f)\circ r`\bT(g)\circ r`\bT(f)\circ s`\bT(f)\circ r;\bT(\alpha)_r`\bT(f)(\gamma)`\bT(g)(\gamma)`\bT(\alpha)_s]
 \end{picture}
\end{center}
follows from MEL, where
\begin{center} \xext=2500 \yext=200
\begin{picture}(\xext,\yext)(\xoff,\yoff)
  \putmorphism(0,60)(1,0)[X `\bT(i)`]{800}{0}a
    \putmorphism(0,150)(1,0)[\phantom{X} `\phantom{\bT(i)}`r]{800}{1}a
    \putmorphism(0,0)(1,0)[\phantom{X} `\phantom{\bT(i)}`s]{800}{1}b
     \put(400,50){$\gamma\Da$}

      \putmorphism(800,60)(1,0)[\phantom{\bT(i)}`\bT(j)`]{1000}{0}a
    \putmorphism(800,150)(1,0)[\phantom{\bT(i)} `\phantom{\bT(j)}`\bT(f)]{1000}{1}a
    \putmorphism(800,0)(1,0)[\phantom{\bT(i)} `\phantom{\bT(j)}`\bT(g)]{1000}{1}b
     \put(1300,50){$\bT(\alpha)\Da$}
\end{picture}
\end{center}
This ends the definition of the $2$-functor $\cD(X,\bT)$.

\subsection*{The $2$-functor of weighted cones}

Using the above $2$-functor(s) we can form the $2$-functor $Cone_W(\bT)$ of $W$-cones over $\bT$.
\begin{center} \xext=2200 \yext=1050
\begin{picture}(\xext,\yext)(\xoff,\yoff)
\setsqparms[0`0`0`0;1600`600]
\putmorphism(150,900)(1,0)[ \cD^{op}`\bCat`Cone_W(\bT)]{1550}{1}a

 \putsquare(150,50)[X`Cone_W(\bT)(X)`Y`Cone_W(\bT)(Y);```]
 \putmorphism(0,650)(0,-1)[\phantom{X}`\phantom{Y}`F]{600}{1}l
  \putmorphism(300,650)(0,-1)[\phantom{X}`\phantom{Y}`G]{600}{1}r
    \put(130,305){$\beta$}
    \put(80,205){$\Longrightarrow$}
    \putmorphism(1500,650)(0,-1)[\phantom{Cone_W(\bT)(X)}`\phantom{Cone_W(\bT)(Y)}`_{Cone_W(\bT)(F)}]{600}{-1}l
  \putmorphism(2020,650)(0,-1)[\phantom{Cone_W(\bT)(X)}`\phantom{Cone_W(\bT)(Y)}`_{Cone_W(\bT)(G)}]{600}{-1}r
    \put(1520,305){$_{Cone_W(\bT)(\beta)}$}
    \put(1660,195){$\Longrightarrow$}

  \put(550,350){\line(0,1){100}}
  \put(550,400){\vector(1,0){360}}
 \end{picture}
\end{center}

Fix $X$ in $\cD$. The category $Cone_W(\bT)(X)$ consists of $2$-natural transformations between $2$-functors $W$ and $\cD(X,\bT)$
and modifications between them. 

    The objects in the category $Cone_W(\bT)(X)$ are $2$-natural transformations
\begin{center} \xext=3000 \yext=1050
\begin{picture}(\xext,\yext)(\xoff,\yoff)

\putmorphism(0,550)(0,-1)[\phantom{i}`\phantom{j}`\cI\ni f]{400}{1}l
  \putmorphism(200,550)(0,-1)[\phantom{i}`\phantom{j}`g]{400}{1}r
    \putmorphism(100,550)(0,-1)[i`j`]{400}{0}r
       \put(60,405){$_{\alpha}$}
    \put(60,295){$\Rightarrow$}

\setsqparms[1`0`0`1;1600`600]
\putmorphism(950,900)(1,0)[ W`\cD(X,\bT)`\lambda]{1550}{1}a

 \putsquare(950,50)[W_i`\cD(X,\bT(i))`W_j`\cD(X,\bT(j));\lambda_i```\lambda_j]
 \putmorphism(800,650)(0,-1)[\phantom{W_i}`\phantom{W_j}`W_f]{600}{1}l
  \putmorphism(1100,650)(0,-1)[\phantom{W_i}`\phantom{W_j}`W_g]{600}{1}r
    \put(930,305){$W_\alpha$}
    \put(880,205){$\Longrightarrow$}
    \putmorphism(2300,650)(0,-1)[\phantom{\cD(X,\bT(i))}`\phantom{\cD(X,\bT(j))}`_{\cD(X,\bT(f))}]{600}{1}l
  \putmorphism(2720,650)(0,-1)[\phantom{\cD(X,\bT(i))}`\phantom{\cD(X,\bT(j))}`_{\cD(X,\bT(g))}]{600}{1}r
    \put(2320,305){$_{\cD(X,\bT(\alpha))}$}
    \put(2460,195){$\Longrightarrow$}
 \end{picture}
\end{center}
so that
\[ \cD(X,\bT(\alpha))\circ \lambda_i = \lambda_j\circ W_\alpha \]

The morphisms in the category $Cone_W(\bT)(X)$ are modifications $\nu:\lambda\ra \lambda'$ or
\begin{center} \xext=800 \yext=250
\begin{picture}(\xext,\yext)(\xoff,\yoff)
  \putmorphism(0,60)(1,0)[W `\cD(X,\bT)`]{800}{0}a
    \putmorphism(0,150)(1,0)[\phantom{W} `\phantom{\cD(X,\bT)}`\lambda]{800}{1}a
    \putmorphism(0,0)(1,0)[\phantom{W} `\phantom{\cD(X,\bT)}`\lambda']{800}{1}b
     \put(360,50){$\nu\Da$}
\end{picture}
\end{center}
such that, for $f:i\ra j$ in $\cI$, the square
\begin{center} \xext=1750 \yext=950
\begin{picture}(\xext,\yext)(\xoff,\yoff)
\setsqparms[0`1`1`0;1600`700]
 \putsquare(150,120)[W_i`\cD(X,\bT(i))`W_j`\cD(X,\bT(j));`_{W_f}`_{\cD(X,\bT(f))}`]
  \putmorphism(150,200)(1,0)[\phantom{W_j}`\phantom{\cD(X,\bT(j))}`\lambda_j]{1600}{1}a
  \putmorphism(150,50)(1,0)[\phantom{W_j}`\phantom{\cD(X,\bT(j))}`\lambda'_j]{1600}{1}b
  \put(1100,120){$_{\nu_j\,\Downarrow}$}

 \putmorphism(150,900)(1,0)[\phantom{W_i}`\phantom{\cD(X,\bT(i))}`\lambda_i]{1600}{1}a
 \putmorphism(150,750)(1,0)[\phantom{W_i}`\phantom{\cD(X,\bT(i))}`\lambda'_i]{1600}{1}b
 \put(1100,820){$_{\nu_i\,\Downarrow}$}

 \end{picture}
\end{center}
commutes, in the sense that
\[ \cD(X,\bT(f)) \circ \nu_i = \nu_j\circ W_f \]
This ends the definition of the category $Cone_W(\bT)(X)$.

The functor
\begin{center} \xext=1200 \yext=150
\begin{picture}(\xext,\yext)(\xoff,\yoff)
  \putmorphism(0,0)(1,0)[Cone_W(\bT)(X)`Cone_W(\bT)(Y)`Cone_W(\bT)(F)]{1600}{1}a
\end{picture}
\end{center}
sends the $2$-natural transformation $\lambda$ to the $2$-natural transformation
\begin{center} \xext=2000 \yext=150
\begin{picture}(\xext,\yext)(\xoff,\yoff)
  \putmorphism(0,0)(1,0)[W`\cD(Y,\bT)`\lambda]{800}{1}a
  \putmorphism(800,0)(1,0)[\phantom{\cD(Y,\bT)}`\cD(X,\bT)`\cD(F,\bT)]{1200}{1}a
\end{picture}
\end{center}
such that, for $i$ in $\cI$,
\begin{center} \xext=2000 \yext=150
\begin{picture}(\xext,\yext)(\xoff,\yoff)
  \putmorphism(0,0)(1,0)[W_i`\cD(Y,\bT(i))`\lambda_i]{800}{1}a
  \putmorphism(800,0)(1,0)[\phantom{\cD(Y,\bT(i))}`\cD(X,\bT(i))`\cD(F,\bT(i))]{1200}{1}a
\end{picture}
\end{center}
is a functor such that, for $u:w\ra w'$ in $W_i$, we have a diagram
\begin{center} \xext=1600 \yext=200
\begin{picture}(\xext,\yext)(\xoff,\yoff)
  \putmorphism(0,60)(1,0)[X `Y`F]{600}{1}a
      \putmorphism(600,60)(1,0)[\phantom{Y}`\bT(i)`]{800}{0}a
    \putmorphism(600,150)(1,0)[\phantom{Y} `\phantom{\bT(i)}`\lambda_i(w)]{800}{1}a
    \putmorphism(600,0)(1,0)[\phantom{Y} `\phantom{\bT(i)}`\lambda_i(w')]{800}{1}b
     \put(900,50){$\lambda_i(u)\Da$}
\end{picture}
\end{center}
and the following equations
\[ \cD(F,\bT(i))\circ \lambda_i(w) = \lambda_i(w)\circ F \]
\[  \cD(F,\bT(i))\circ \lambda_i(u) = \lambda_i(u)_F  \]
hold. Moreover, the functor $Cone_W(\bT)(F)$ sends the modification $\nu$
\begin{center} \xext=800 \yext=200
\begin{picture}(\xext,\yext)(\xoff,\yoff)
      \putmorphism(0,60)(1,0)[W`\cD(Y,\bT)`]{800}{0}a
    \putmorphism(0,150)(1,0)[\phantom{W} `\phantom{\cD(Y,\bT)}`\lambda]{800}{1}a
    \putmorphism(0,0)(1,0)[\phantom{W} `\phantom{\cD(Y,\bT)}`\lambda']{800}{1}b
     \put(300,50){$\nu\Da$}
\end{picture}
\end{center}
to the modification
\begin{center} \xext=2000 \yext=250
\begin{picture}(\xext,\yext)(\xoff,\yoff)
      \putmorphism(0,60)(1,0)[W`\cD(X,\bT)`]{2000}{0}a
    \putmorphism(0,150)(1,0)[\phantom{W} `\phantom{\cD(X,\bT)}`Cone_W(\bT)(F)(\lambda)=\bar{\lambda}]{2000}{1}a
    \putmorphism(0,0)(1,0)[\phantom{W} `\phantom{\cD(X,\bT)}`Cone_W(\bT)(F)(\lambda')=\bar{\lambda'}]{2000}{1}b
     \put(300,50){$Cone_W(\bT)(F)(\nu)=\bar{\nu}\Da$}
\end{picture}
\end{center}
such that, for $i$ in $\cI$,
\begin{center} \xext=1000 \yext=250
\begin{picture}(\xext,\yext)(\xoff,\yoff)
      \putmorphism(0,60)(1,0)[W`\cD(X,\bT(i))`]{1000}{0}a
    \putmorphism(0,150)(1,0)[\phantom{W} `\phantom{\cD(X,\bT(i))}`\bar{\lambda}_i]{1000}{1}a
    \putmorphism(0,0)(1,0)[\phantom{W} `\phantom{\cD(X,\bT(i))}`\bar{\lambda'}_i]{1000}{1}b
     \put(400,50){$\bar{\nu}_i\Da$}
\end{picture}
\end{center}
is a natural transformation such, that for $w$ in $W_i$,
\begin{center} \xext=2000 \yext=150
\begin{picture}(\xext,\yext)(\xoff,\yoff)
  \putmorphism(0,0)(1,0)[\bar{\lambda}_i(w)=\lambda_i(w)\circ F`\lambda'_i(w)\circ F=\bar{\lambda'}_i(w)`(\bar{\nu}_i)_w)=((\nu_i)_w)_F]{2000}{1}a
\end{picture}
\end{center}
is a morphism in $\cD(X,\bT(i))$.

The component, at the $2$-natural transformation  $\lambda: W\ra \cD(X,\bT)$, of the natural transformation
\begin{center} \xext=1800 \yext=150
\begin{picture}(\xext,\yext)(\xoff,\yoff)
  \putmorphism(0,0)(1,0)[Cone_W(\bT)(F)`Cone_W(\bT)(G)`Cone_W(\bT)(\beta)]{1800}{1}a
\end{picture}
\end{center}
is a modification $\cD(\beta,\bT)\circ \lambda$, i.e. the composition
\begin{center} \xext=2000 \yext=320
\begin{picture}(\xext,\yext)(\xoff,\yoff)
  \putmorphism(0,110)(1,0)[W`\cD(Y,\bT)`\lambda]{800}{1}a
      \putmorphism(800,110)(1,0)[\phantom{\cD(Y,\bT)}`\cD(X,\cT)`]{1200}{0}a
    \putmorphism(800,200)(1,0)[\phantom{\cD(Y,\bT)} `\phantom{\cD(X,\bT)}`\cD(F,\bT)]{1200}{1}a
    \putmorphism(800,50)(1,0)[\phantom{\cD(Y,\bT)} `\phantom{\cD(X,\bT)}`\cD(G,\bT)]{1200}{1}b
     \put(1100,100){$\cD(\beta,\bT)\Da$}
\end{picture}
\end{center}
so that, at $i$ in $\cI$, it is the natural transformation $\cD(\beta,\bT(i))\circ \lambda_i$
\begin{center} \xext=2200 \yext=320
\begin{picture}(\xext,\yext)(\xoff,\yoff)
  \putmorphism(0,110)(1,0)[W_i`\cD(Y,\bT(i))`\lambda_i]{800}{1}a
      \putmorphism(800,110)(1,0)[\phantom{\cD(Y,\bT(i))}`\cD(X,\bT(i))`]{1400}{0}a
    \putmorphism(800,200)(1,0)[\phantom{\cD(Y,\bT(i))} `\phantom{\cD(X,\bT(i))}`\cD(F,\bT(i))]{1400}{1}a
    \putmorphism(800,50)(1,0)[\phantom{\cD(Y,\bT(i))} `\phantom{\cD(X,\bT(i))}`\cD(G,\bT(i))]{1400}{1}b
     \put(1100,100){$\cD(\beta,\bT(i))\Da$}
\end{picture}
\end{center}
so that, for $w$ in $W_i$, it is a morphism in $\cD(X,\bT)$
\begin{center} \xext=1200 \yext=150
\begin{picture}(\xext,\yext)(\xoff,\yoff)
  \putmorphism(0,0)(1,0)[\lambda_i(w)\circ F`\lambda_i(w)\circ G`\lambda_i(w)(\beta)]{1200}{1}a
\end{picture}
\end{center}
from the diagram
\begin{center} \xext=1600 \yext=200
\begin{picture}(\xext,\yext)(\xoff,\yoff)
  \putmorphism(800,60)(1,0)[\phantom{Y}`\bT(i)`\lambda_i(w)]{600}{1}a
      \putmorphism(0,60)(1,0)[X`Y`]{800}{0}a
    \putmorphism(0,150)(1,0)[\phantom{X} `\phantom{Y}`F]{800}{1}a
    \putmorphism(0,0)(1,0)[\phantom{X} `\phantom{Y}`G]{800}{1}b
     \put(300,50){$\beta\Da$}
\end{picture}
\end{center}

\subsection*{The representation of the $2$-functor $Cone_W(\bT)$}

The representation of the functor $Cone_W(\bT)$ is the $W$-weighted limit of the $2$-functor $\bT$.
Thus it is an object $Lim_W(\bT)$ together with a $2$-natural isomorphism
\begin{center} \xext=1200 \yext=150
\begin{picture}(\xext,\yext)(\xoff,\yoff)
  \putmorphism(0,0)(1,0)[\cD(-,Lim_W(\bT))`Cone_W(\bT)`\varrho]{1200}{1}a
\end{picture}
\end{center}
The image of the identity on $Lim_W(\bT)$ is the limiting $W$-weighted cone
\begin{center} \xext=800 \yext=150
\begin{picture}(\xext,\yext)(\xoff,\yoff)
  \putmorphism(0,0)(1,0)[Lim_W(\bT)`\bT`\pi]{800}{1}a
\end{picture}
\end{center}
in  $Cone_W(\bT)(Lim_W(\bT))$. For any $0$-cell $X$ we have a correspondence via $\pi$
\begin{center} \xext=2200 \yext=750
\begin{picture}(\xext,\yext)(\xoff,\yoff)

     \putmorphism(200,560)(1,0)[X`Lim_W(\bT)`]{1000}{0}a
    \putmorphism(200,650)(1,0)[\phantom{X} `\phantom{Lim_W(\bT)}`L]{1000}{1}a
    \putmorphism(200,500)(1,0)[\phantom{X} `\phantom{Lim_W(\bT)}`L']{1000}{1}b
     \put(500,550){$n\,\Da$}

      \putmorphism(80,520)(0,-1)[\phantom{X}`\phantom{\bT}`\lambda]{550}{1}l
       \putmorphism(240,520)(0,-1)[\phantom{X}`\phantom{\bT}`\lambda']{550}{1}r
        \putmorphism(150,520)(0,-1)[\phantom{X}`\bT`]{550}{0}l
          \put(1150,480){\vector(-2,-1){870}}

\put(700,200){${\pi}$}

\put(120,250){${\nu}$}
\put(120,170){${\Rightarrow}$}
\end{picture}
\end{center}
or in another form, we have an isomorphism of categories
\begin{center} \xext=2200 \yext=750
\begin{picture}(\xext,\yext)(\xoff,\yoff)

     \putmorphism(200,560)(1,0)[X`Lim_W(\bT)`]{1000}{0}a
    \putmorphism(200,650)(1,0)[\phantom{X} `\phantom{Lim_W(\bT)}`L]{1000}{1}a
    \putmorphism(200,500)(1,0)[\phantom{X} `\phantom{Lim_W(\bT)}`L']{1000}{1}b
     \put(500,550){$n\,\Da$}

      \put(1700,550){${\rm in} \;\cD$}

  \put(0,350){\line(1,0){2000}}
      \putmorphism(200,60)(1,0)[X`\bT`]{1000}{0}a
    \putmorphism(200,150)(1,0)[\phantom{X} `\phantom{\bT}`\lambda]{1000}{1}a
    \putmorphism(200,0)(1,0)[\phantom{X} `\phantom{\bT}`\lambda']{1000}{1}b
     \put(500,50){$\nu\,\Da$}

       \put(1700,50){${\rm in}\; Cone_W(\bT)(X)$}
\end{picture}
\end{center}
natural in $X$.

\end{document}